\documentclass[leqno,a4paper,11pt]{article}
\title{Global existence and blow-up of  solutions to some quasilinear wave equation in one space dimension}
 \author{YUUSUKE SUGIYAMA\footnote{sugiyama@ma.kagu.tus.ac.jp,\ \mbox{Department of Mathematics,\ Tokyo University of Science}}}
\date{}
\usepackage{amssymb,amsmath,amsthm,mathrsfs,delarray,enumerate}%mathrsfs
\usepackage{graphics}

\theoremstyle{definition} 
\newtheorem{Def}{deffinition}
\newtheorem{Prop}[Def]{Proposition}
\newtheorem{Lemma}[Def]{Lemma}
\newtheorem{Thm}[Def]{Theorem}
\newtheorem{Remark}[Def]{Remark}

\newcommand{\R}{\mathbb{R}} %set of real number
\makeatletter

\@addtoreset{equation}{section}

\begin{document}
%%%%%%%%%%%%%%%%  make title and   %%%%%%%%%%%%%%%%
\maketitle %タイトルをここに出力する
%%%%%%%%%%%%%%%%  abstract   %%%%%%%%%%%%%%%%
\begin{abstract}
We consider the global existence and blow up  of solutions of the Cauchy problem of the  quasilinear wave equation:
$\partial_{t}^2 u =  \partial_x( c(u)^2 \partial_x u)$, which has richly physical
backgrounds. 
Under the assumption that $c(u(0,x))\geq \delta$ for some $\delta>0$, we give sufficient conditions for the existence of global smooth solutions  and  the occurrence of two types of blow-up  respectively. One of the two types is that $L^{\infty}$-norm of $\partial_t u$ or $\partial_x u$ goes up to the infinity. The other type is that $c(u)$ vanishes, that is, the equation degenerates.
\end{abstract}
%%%%%%%%%%%%%%%%  text start  %%%%%%%%%%%%%%%%
%改行時に字下げしない。

\section{Introduction}
In this paper,\ we consider the Cauchy problem of the following wave equation: 
\begin{eqnarray} 
\left\{  \begin{array}{ll} \label{eq1}
  \partial_{t}^2 u =  \partial_x( c(u)^2 \partial_x u), \ \ (t,x) \in (0,T] \times \R, \\   
  u(0,x) = u_0 (x),\ \ x \in \R,                \\              
    \partial_t u(0,x) = u_1 (x), \ \ x \in \R, 
\end{array} \right. 
\end{eqnarray}
where $u(t,x)$ is an unknown real valued function. The equation in (\ref{eq1}) has some physical backgrounds including  vibrations of a  string.

We assume that $c \in C^{\infty}((\theta_0, \infty ))$ for some $\theta_0 \in [-\infty, 0)$  satisfies that
\begin{eqnarray} 
&& \lim_{\theta \searrow \theta_0} c(\theta)=0 , \label{con3} \\
&& c(\theta )>0  \ \  \mbox{for all}  \ \theta>\theta_0 , \label{con4} \\
&& c'(\theta ) \geq 0 \ \  \mbox{for }  \ \theta >  \theta_0. \label{con2} 
\end{eqnarray} 

We denote  Sobolev space $(1-\partial^2 _x)^{-\frac{s}{2}}L^2(\R)$  for $s\in \R$  by $H^s (\R)$. For a Banach space $X$, $C^j ([0,T];X)$  denotes the set of functions $f:[0,T] \rightarrow X$ such that $f(t)$ and its $k$ times derivatives for $k=1,2,\ldots , j$
 are continuous.  $L^{\infty} ([0,T];X)$ denotes the set of functions $f:[0,T] \rightarrow X$ such that 
 the norm $\| f\|_{L^{\infty} ([0,T];X)}:=\mbox{ess.}\sup_{[0,T]}\| f(t)\|_X$ is finite. Various positive constants are simply denoted by $C$.

By dividing the both side of $(\ref{eq1})$ by $c(u(t,x))^2$,\ $(\ref{eq1})$ is formed to  
\begin{eqnarray} \label{st}
\dfrac{1}{c(u(t,x))^2}\partial^2_t u(t,x)  -\partial^2 _x u(t,x) =\dfrac{2c'(u(t,x))(\partial _x u(t,x))^2}{c(u(t,x))}.
\end{eqnarray}
Since the left hand side of $(\ref{st})$ has a singularity at $u=\theta_0$,\ we call a solution $u$ to hyperbolic 
equation $(\ref{eq1})$ to blow up, when 
\begin{eqnarray}\label{type1}
\varlimsup_{t\nearrow T}( \| \partial_t u(t)\|_{L^{\infty}}+\|\partial_x u(t) \|_{L^{\infty}})=\infty,
\end{eqnarray}
or
\begin{eqnarray}\label{type2}
\lim_{t\nearrow T} \inf_{(s,x) \in [0,t]\times \R}u(s,x)=\theta_0,
\end{eqnarray}
 occurs in finite time $T>0$ under the assumption that $u(0,x)\geq \delta$ for some positive constant $\delta$.
The blow up criterion  $(\ref{type1})$ and $(\ref{type2})$ of  some class of hyperbolic systems including $(\ref{eq1})$ is introduced in the textbooks of  Majda \cite{m} and Alinhac \cite{al}. 
 The aim of this paper is to obtain sufficient conditions for the global existence of solutions   and  the occurrence of the blow-up  phenomena $(\ref{type1})$ and $(\ref{type2})$  in finite time  respectively.

We denote the blow up time of the  solution $u$ of the Cauchy problem  (\ref{eq1})  by $T^{*} $,
that is, 
\begin{eqnarray*}
T^{*} :=\sup \{ \ T> 0 \ | \ \sup_{[0,T]} \{ \| \partial_t u(t)\|_{L^{\infty}}+\|\partial_x u(t) \|_{L^{\infty}} \} <\infty ,\ \inf_{[0,T]\times \R} u(t,x) >\theta_0 \ \}.
\end{eqnarray*}

\begin{Thm}\label{mt1} \it
Let  $c(\cdot ) \in C^{\infty}((\theta_0, \infty))$  and initial data $(u_0, u_1) \in H^{s+1}(\R) \times H^{s}(\R)$ for $s>\frac{1}{2}$. Suppose $c(\theta)$ and $(u_0, u_1)$ satisfy  $(\ref{con3})$, $(\ref{con4})$,  $(\ref{con2})$ and 
\begin{eqnarray} 
&& u_0 (x)> \theta_0  \ \ \mbox{for} \  x \in \R ,\label{inicon1} \\
&&u_1 (x) \pm c(u_0 (x)) \partial_x u_0 (x) \leq 0 \ \ \mbox{for} \  x \in \R ,\label{inicon2} \\
&&  -\int_{\R} u_1(x) dx<\int_{\theta_0} ^{0} c(\theta ) d\theta .\label{inicon3} 
\end{eqnarray} 
Then $(\ref{eq1})$ has a unique global solution such that 
$
u \in \bigcap_{j=0,1,2} C^{j}([0, \infty);H^{s-j+1} (\R)).
$
\end{Thm}

\begin{Thm}\label{mt2} \it
Let $\theta_0 \not=-\infty$. Under the same assumption as in Theorem $\ref{mt1}$ without $(\ref{inicon3})$, we assume that
\begin{eqnarray} 
&&\mbox{\rm supp\it } \, u_0,\ \mbox{\rm supp\it} \, u_1 \subset [-K,K] \ \ \mbox{for some} \ K>0, \label{inicon4}\\
&& -\int_{\R} u_1(x) dx  > -2\theta_0 c(0) .\label{inicon5}  
\end{eqnarray} 
 Then  $T^{*} < \infty$ and  the solution  $u \in \bigcap_{j=0,1,2} C^{j}([0, T^*);H^{s-j+1} (\R))$ of $(\ref{eq1})$ satisfies that  
\begin{eqnarray*}
\lim_{t\nearrow T^{*} } u(t,x_0)=\theta_0  \ \ \mbox{for some} \ x_0 \in \R.
\end{eqnarray*}
\end{Thm}

 \begin{Thm}\label{mt3} \it
Let  $c \in C^{\infty}((\theta_0, \infty))$  and initial data $(u_0, u_1) \in H^{s+1}(\R) \times H^{s}(\R)\setminus \{ 0 \}$ for $s>\frac{1}{2}$. Suppose $c(\theta)$ and $(u_0, u_1)$ satisfy that there exists a constant $\delta>0$ such that
\begin{eqnarray} 
&& u_0 (x)> \theta_0 \ \ \mbox{for} \  x \in \R ,\label{inicon1-2} \\
&& c'(\theta)>0  \ \  \mbox{for all}  \ \theta > \theta_0, \label{inicon6}\\
&&\mbox{\rm supp\it } \, u_0,\ \mbox{\rm supp\it} \, u_1 \subset [-K,K] \ \ \mbox{for some} \ K>0, \label{inicon7}\\
&&u_1 (x) \pm c(u_0 (x)) \partial_x u_0 (x) \geq 0 \ \ \mbox{for} \  x \in \R .\label{inicon8} 
\end{eqnarray} 
 Then  $T^{*}< \infty$ and  the solution  $u \in \bigcap_{j=0,1,2} C^{j}([0, T^*);H^{s-j+1} (\R))$ of $(\ref{eq1})$ satisfies
\begin{eqnarray*}
\varlimsup_{t\nearrow T^{*}} \| \partial_t u(t)\|_{L^{\infty}}+\|\partial_x u(t) \|_{L^{\infty}}=\infty.
\end{eqnarray*}
\end{Thm} 

% \begin{Remark} \label{rem0}
% In Theorems \ref{mt1} and \ref{mt2},  from  (\ref{con4}),  (\ref{con2}) and (\ref{inicon2}), we note
% $$
% \int_{\R} u_1 (x) dx \leq 0 \ \ \mbox{and} \ \
% -2\theta_0 c(0) > \int_{\theta_0} ^{0} c(\theta ) d\theta .
% $$
% \end{Remark}

\begin{Remark} \label{rem1}
Let  $c(\cdot  ) \in C^{\infty}(\R)$  and initial data $(u_0, u_1) \in H^{s+1}(\R) \times H^{s}(\R)\setminus \{ 0 \}$ for $s>\frac{1}{2}$.
Suppose that there exists a constant $c_1 >0$ such that  $c(\theta)\geq c_1$ for all $\theta \in \R$ instead of the assumption   $(\ref{con3})$ and $(\ref{con4})$.
If  $(\ref{con2})$ and $(\ref{inicon2})$ hold, then $(\ref{eq1})$ has a unique global solution such that  $u \in \bigcap_{j=0,1,2} C^{j}([0, \infty);H^{s-j+1} (\R))$.
\end{Remark}

\begin{Remark} \label{rem3}
In Theorem \ref{mt1},  if $\theta = -\infty$, then we does not need the assumption (\ref{inicon3}).
\end{Remark}

\begin{Remark} \label{rem2}
The equation in (\ref{eq1}) does not degenerate for the global solution which is constructed by Theorem \ref{mt1}, that is,
the global solution $u$ in Theorem \ref{mt1} satisfies that there exists a constant $\theta_1 >\theta_0$ such that
$$
u(t,x)\geq \theta_1,
$$
for $(t,x) \in [0,\infty)\times \R$.
\end{Remark}
The equation in (\ref{eq1}) has richly physical backgrounds (e.g. the flow of a one dimensional gas, the shallow water waves, the longitudinal wave propagation on a moving threadline,  the dynamics of a finite nonlinear string,  the elastic-plastic materials or  the electromagnetic transmission line). 
  In \cite{al}, Ames, Lohner and Adams study the group  properties of the  equation in (\ref{eq1}) by using the Lie algebra and introduce physical backgrounds.
In \cite{z}, Zabusky introduce the equation 
\begin{eqnarray}  \label{zeqori}
\partial_{t}^2 v =   (1+ \partial_x v)^a \partial^2 _x v,
\end{eqnarray} 
which describes the standing vibrations of a finite, continuous and nonlinear string for $a>0.$ Setting $u=\partial_x v$ for the solution $v$ to (\ref{zeqori}),
$u$ is a solution to the equation: 
\begin{eqnarray}  \label{zeq}
\partial_{t}^2 u =  \partial_x ( (1+  u)^a \partial_x u).
\end{eqnarray} 

In author's previous work \cite{s3},  the author show a global existence theorem for (\ref{eq1}) under some conditions on the function $c$ and initial data.
However, we can not apply the global existence theorem of \cite{s3} to (\ref{zeq}) since the theorem requires  the condition that there exists a constant $c_0>0$ such that 
\begin{eqnarray} \label{pas}
c(\theta)\geq c_0 \ \ \mbox{for all} \ \theta \in \R.
\end{eqnarray} 
Our global existence theorem (Theorem 1)  can yield a global solvability for
some equations including (\ref{zeq}).

Many authors \cite{GHZ, GHZ2, zz1, zz22, ks1, ks2, s3} study the Cauchy problem of the equation
\begin{eqnarray} \label{raeq}
 \partial_{t}^2 u =  c(u)^2 \partial^2 _x u + \lambda c(u)c'(u)(\partial_x u)^2,
\end{eqnarray}
for $ 0\leq \lambda \leq 2$. (\ref{raeq}) with $ \lambda =2$ is the equation in (\ref{eq1}).

Kato and Sugiyama \cite{ks2} and Sugiyama \cite{ks2} show that 
the same theorem as Theorem \ref{mt2} holds for (\ref{raeq}) for $0\leq \lambda  < 2 $ without 
the restriction $\int_{\R} u_1 (x) dx$ (the assumption (\ref{inicon5})). 

The equation in $(\ref{eq1})$ is related to equations 
\begin{eqnarray} \label{eq4}
\partial_t v =\pm c(v) \partial_x v \ \ \mbox{and} \ \  
\partial^2 _t v=c(\partial_x v) \partial^2 _x v.
\end{eqnarray}
In fact, the solution $v$ to the first equation of (\ref{eq4}) is a solution to the equation in   (\ref{eq1}). The function $\partial_x v$ with the 
solution $v$ to the second equation of (\ref{eq4}) is a solution to  the equation in   (\ref{eq1}).  Lax \cite{lax} and John \cite{fj} study the blow up for the first and the second equations of (\ref{eq4}) respectively. In \cite{mm}, MacCamy and Mizel study the  Dirichlet problem for the second equation in (\ref{eq4}).

The blow up of the $2$ and $3$ dimensional versions of the equation in (\ref{eq1}):
\begin{eqnarray*}
\partial^2 _t u=\mbox{div} (c(u)^2 \nabla u), 
\end{eqnarray*}
is studied by  Li, Witt and Yin \cite{LWY} and Ding and Yin \cite{dy} respectively.

We prove  Theorem \ref{mt1} by using Zhang and Zheng's idea in \cite{zz1}  and an  estimate which ensure that the equation does not degenerate. In \cite{zz1}, Zhang and Zheng show the global existence of solution to (\ref{raeq}) with $\lambda =1$ under some conditions on $c$ and initial data including (\ref{pas}). 

The proof of Theorem \ref{mt2} is based on the method in \cite{ks2, s3} which give a sufficient condition that the equation (\ref{raeq}) for  $0\leq \lambda  < 2 $  and $c(u)=u+1$ degenerates in finite time.

In the proof of Theorem \ref{mt3}, we use the Riemann invariants and the method of characteristic.

This paper is organized as follows: In Section 2, we introduce the local existence and the uniqueness of solutions of (\ref{eq1}).
In Sections 3, 4 and 5, we show  Theorems \ref{mt1}, \ref{mt2} and \ref{mt3} respectively.

\section{Local existence and uniqueness}
In this section,\ we introduce the local existence and the uniqueness of solutions of $(\ref{eq1})$.
The local well-posedness of some class of second order quasilinear strictly hyperbolic equations including the equation $(\ref{eq1})$ is established  by Hughes, Kato and  Marsden \cite{HKM}. Their proofs are based on the Energy method.\ Furthermore,\ by the Moser type inequality,\ the above local well-posedness results are sharpened (e.g.  Majda \cite{m} and  Taylor \cite{tay1}).\ Roughly speaking,\ the results in \cite{m} and  \cite{tay1} state that the solution $u$ of $(\ref{eq1})$ persists as long as $\|\partial_t u\|_{L^{\infty}}$ and $\|\partial_x  u\|_{L^{\infty}}$ are bounded.\ 

The following theorem is obtained by applying Theorem  $2.2$ in \cite{m} and Proposition $5.3.B$ in \cite{tay1} to the Cauchy problem $(\ref{eq1})$. 
\begin{Prop} \label{lwp} \it
Suppose that  $c(\theta)$ and $(u_0,u_1) \in H^{s+1} (\R) \times H^s (\R)$ for  $s > \frac{1}{2}$ and $c \in C^{\infty} (\R)$ satisfy $(\ref{inicon1})$.
Then there exist $T>0$ and  a unique solution $u$ of  \rm (\ref{eq1}) \it with
\begin{eqnarray} \label{cls1}
u\in \bigcap_{j=0,1,2} C^{j}([0,\infty);H^{s-j+1} (\R)) 
\end{eqnarray}
and 
\begin{eqnarray} \label{cls2}
u(t,x)> \theta_0 \ \ \mbox{for} \ (t,x) \in [0,T]\times \R.
\end{eqnarray}

Furthermore,\ if $(\ref{eq1})$ does not have a global solution $u$ satisfying $(\ref{cls1})$ and $(\ref{cls2})$,\ then the solution $u$ satisfies 
\begin{eqnarray}  \label{b1}
\varlimsup_{t\nearrow T} \| \partial_t u(t)\|_{L^{\infty}}+\|\partial_x u(t) \|_{L^{\infty}}=\infty .
\end{eqnarray}
or
\begin{eqnarray} \label{b2}
\lim_{t\nearrow T} \inf_{(s,y)\in [0,t)\times \R} u(s,y)=\theta_0,
\end{eqnarray}
for some $T>0$.
\rm
\end{Prop}

\section{Proof of Theorem $\ref{mt1}$}

We set the Riemann invariants $R_1 (t,x)$ and $R_2(t,x)$ as follows
\begin{eqnarray}\label{ri}
\begin{array}{ll}
R_1=\partial_t u +c(u)\partial_x u, \\
R_2=\partial_t u -c(u)\partial_x u.
\end{array}
\end{eqnarray}
By $(\ref{eq1})$, $R_1$ and $R_2$ are  solutions to the system of the following first oder equations
\begin{eqnarray}\label{fs}
 \left\{  
\begin{array}{ll}
\partial_t R_1-c(u)\partial_x R_1=\dfrac{c'(u)}{2c(u)}(R_1^2-R_2R_1), \\
\partial_t u = \dfrac{1}{2}(R_1+R_2),\\
\partial_t R_2+c(u)\partial_x R_2=\dfrac{c'(u)}{2c(u)}(R_2^2-R_1R_2). 
\end{array}
\right.
\end{eqnarray}  
For the proof of Theorem \ref{mt1}, we prove some lemma.
\begin{Lemma}\label{zz1} \it 
Suppose that
$c(\theta ) \in C^{\infty}((\theta_0, \infty ))$  and initial data $(u_0, u_1) \in H^{s+1}(\R) \times H^{s}(\R)$ with $\displaystyle s>\frac{1}{2}$
satisfy  $(\ref{inicon1})$ and that  $R_1$ and $R_2$ are the functions in $(\ref{ri})$ for the solution $u$ of $(\ref{eq1})$ such that 
$\displaystyle
u \in \bigcap_{j=0,1,2} C^{j}([0,T^{*});H^{s-j+1} (\R)).
$

If  $R_1 (0,x )\geq  0$  for all $x$, then
$R_1 (t, x)\geq 0$  for all $(t,x) \in [0,T^{*})\times \R$.

If  $R_1 (0,x )\leq   0$  for all $x$, then
$R_1 (t, x)\leq 0$  for all $(t,x) \in [0,T^{*})\times \R$.

If  $R_2 (0,x )\geq  0$  for all $x$, then
$R_2 (t, x)\geq 0$  for all $(t,x) \in [0,T^{*})\times \R$. 

If  $R_2 (0,x )\leq   0$  for all $x$, then
$R_2 (t, x)\leq 0$  for all $(t,x) \in [0,T^{*})\times \R$.

\rm
\end{Lemma} 

\begin{proof}

We show that  $R_1(t,\cdot )\geq 0$ with $R_1(0,0)\geq 0$ only.

For any point $(t_0,x_0) \in [0,T]\times \R$, let  $x_{\pm} (t)$ denote the plus and minus characteristic curves on the  first and third equations of $(\ref{fs})$ through $(t_0,x_0)$ respectively as follows,
\begin{eqnarray} \label{chl}
\frac{dx_{\pm} (t)}{dt}=\pm u(t,x_{\pm} (t)),\ \ \ \ x_{\pm} (t_0) =x_0 .
\end{eqnarray}
From $(\ref{fs})$, $R_1(t, x_{-}(t))$ is a solution to  
\begin{eqnarray} \label{rd}
\dfrac{d}{dt} R_1(t, x_{-}(t))=\dfrac{c'(u)}{2c(u)}(R_1 (t, x_{-}(t))^2-R_2 (t, x_{-}(t)) R_1(t, x_{-}(t))).
\end{eqnarray}
By the uniqueness of the differential equation (\ref{rd}), 
 we have $R_1(t, x_{-}(t))=0$ for $t \in [0,T^{*}) $ with $R_1(0, x_{-}(0))=0$, which implies that
 $R_1(t, \cdot )\geq 0$ with $R_1(0,\cdot )\geq 0$.

\end{proof}

\begin{Lemma}\label{zz2} \it 
Let $p\in [1,\infty)$. Suppose that $c(\theta ) \in C^{\infty}((\theta_0, \infty ))$  and initial data $(u_0, u_1) \in H^{s+1}(\R) \times H^{s}(\R)$ with  $\displaystyle s>\frac{1}{2}$
satisfy    $(\ref{con2})$,  $(\ref{inicon1})$  and  $(\ref{inicon2})$. Then  we  have
\begin{eqnarray}\label{lemmazzes2} 
\| R_1 (t) \|^p _{L^{p}}+\| R_2(t)\|^p _{L^{p}}\leq \|R_1 (0) \|^p _{L^{p}}+\| R_2(0) \|^p _{L^{p}}, \ \mbox{for} \ t \in [0,T^* ) ,
\end{eqnarray}
where  $R_1$ and $R_2$ are the functions in $(\ref{ri})$ for the solution $u$ of $(\ref{eq1})$ such that 
$\displaystyle
u \in \bigcap_{j=0,1,2} C^{j}([0,T^{*});H^{s-j+1} (\R)).
$
\rm
\end{Lemma} 

\begin{proof}
The proof is almost the  same as in the proof of Lemma 5 in  Zhang and  Zheng's paper \cite{zz1}. We give the proof of this lemma for reader's convenience.

We denote $\tilde{R_1}:=-R_1$ and  $\tilde{R_2}:=-R_2$. 
Lemma $ \ref{zz1}$ implies that $\tilde{R_1}(t)\geq 0$ and $\tilde{R_2}(t) \geq 0$ for all $t$.
 By the first equation of $(\ref{fs})$, we have
$$ 
\partial_t \tilde{R_1}-c(u)\partial_x \tilde{R_1}=-\dfrac{c'(u)}{2c(u)}(\tilde{R_1}^2-\tilde{R_2}\tilde{R_1}).
$$
Multiplying the both side of the above equation by $(\tilde{R_1})^{p-1}$, we obtain
\begin{eqnarray}\label{zz2es1} 
\frac{1}{p}\{ \partial_t (\tilde{R_1})^p  -c \partial_x (\tilde{R_1})^p\}=-\frac{c'}{2c}((\tilde{R_1})^{p+1}-\tilde{R_2}(\tilde{R_1})^p),
\end{eqnarray}
By the third equation of $(\ref{fs})$, we have
\begin{eqnarray}\label{zz2es2} 
\frac{1}{p} c \partial_x (\tilde{R_1})^p=\frac{1}{p} \partial_x (c(\tilde{R_1})^p)+\frac{1}{p}\frac{c'}{2c}(\tilde{R_1}-\tilde{R_2}),
\end{eqnarray}
from which, $(\ref{zz2es1})$ yields that
\begin{align} 
\frac{1}{p}\{ \partial_t (\tilde{R_1})^p  - \partial_x (c(\tilde{R_1})^p)\}=&-(\frac{1}{2}-\frac{1}{2p})\frac{c'}{c}(\tilde{R_1})^{p+1}\notag \\
&+\frac{c}{2c'}\tilde{R_2}(\tilde{R_1})^p-\frac{c'}{2pc}\tilde{R_2}(\tilde{R_1})^p .\label{zz2es3}
\end{align}
By the similar computation as above, we have
\begin{align}
\frac{1}{p}\{ \partial_t (\tilde{R_2})^p  - \partial_x (c(\tilde{R_2})^p)\}=&-(\frac{1}{2}-\frac{1}{2p})\frac{c'}{c}(\tilde{R_2})^{p+1} \notag \\
&+\frac{c'}{2c}\tilde{R_1}(\tilde{R_2})^p-\frac{c'}{2pc}\tilde{R_1}(\tilde{R_2})^p.  \label{zz2es4} 
\end{align}
By summing up  $(\ref{zz2es3})$ and  $(\ref{zz2es4})$ and integration over $\R$, we have
\begin{align}\label{zz2es5} 
\frac{1}{p} \frac{d}{dt}  \int_{\R_1}  (\tilde{R_1})^p +  (\tilde{R_2})^p dx=&-(\frac{1}{2}-\frac{1}{2p}) \int_{\R_1} \frac{c'}{c}((\tilde{R_1})^{p+1}-\tilde{R_1}(\tilde{R_2})^p) dx \notag \\
&-(\frac{1}{2}-\frac{1}{2p}) \int_{\R} \frac{c'}{c}((\tilde{R_2})^{p+1}-\tilde{R_2}(\tilde{R_1})^p) dx \notag \\
=& -(\frac{1}{2}-\frac{1}{2p}) \int_{\R} \frac{c'}{c}(\tilde{R_1}-\tilde{R_2})((\tilde{R_1})^{p}-(\tilde{R_2})^p) dx\leq 0.
\end{align}
Therefore, integrating the both side of $(\ref{zz2es5})$ over $[0,t]$, we have $(\ref{lemmazzes2})$.
\end{proof}

\begin{Lemma}\label{zz22} \it 
Under the same assumption as in Lemma \ref{zz2}, we have
\begin{eqnarray}\label{inftyes} 
\| R_1 (t) \| _{L^{\infty}}+\| R_2(t)\| _{L^{\infty}}\leq 2 (\|R_1 (0) \| _{L^{\infty}}+\| R_2(0) \| _{L^{\infty}}), \ \mbox{for} \ t \in [0,T^* ) .
\end{eqnarray}
\rm
\end{Lemma} 

\begin{proof}
Noting inequalities $a^p+b^p \leq (a+b)^p $ and $(a+b)^p \leq 2^p (a+b)^p $ for $a,b\geq 0$, by raising the both side of  (\ref{lemmazzes2}) to the $\dfrac{1}{p}$ power, we have
\begin{eqnarray*} 
\| R_1 (t) \| _{L^{p}}+\| R_2(t)\| _{L^{p}}\leq 2 (\|R_1 (0) \| _{L^{p}}+\| R_2(0) \| _{L^{p}}).
\end{eqnarray*}
From the fact that $\displaystyle \lim_{p \rightarrow \infty} \| u \|_{L^p}=\|u\|_{L^{\infty}} $ with $ u \in H^s (\R)$ $(s> 1/2)$ (e.g. Lemma 11 in \cite{s3}), we have (\ref{inftyes}).

\end{proof}

\begin{Lemma}\label{zz3} \it 
Suppose that $c(\theta ) \in C^{\infty}((\theta_0, \infty ))$  and initial data $(u_0, u_1) \in H^{s+1}(\R) \times H^{s}(\R)$ with $\displaystyle s>\frac{1}{2}$
satisfy   $(\ref{con2})$, $(\ref{inicon1})$, $(\ref{inicon2})$ and $(\ref{inicon3})$. Then  there exists $\theta_1 > \theta_0$ such that
\begin{eqnarray}\label{lemmazzes3} 
u(t,x)\geq \theta_1, \ \ \mbox{for} \ (t,x) \in [0,T^{*}) \times \R,
\end{eqnarray}
where  $R_1$ and $R_2$ are the functions which defined in $(\ref{ri})$ for the solution $u$ of $(\ref{eq1})$ such that 
$
u \in \bigcap_{j=0,1,2} C^{j}([0,T^{*});H^{s-j+1} (\R)).
$
\rm
\end{Lemma} 
\begin{proof}
From Lemma \ref{zz1}, we have  
\begin{eqnarray} \label{p1}
|c(u) \partial_x u(t,x)| \leq - \partial_t u(t,x),
\end{eqnarray}
from which, a simple computation yields that
\begin{align} 
\left|\int_0 ^{u(t,x)} c(\theta) d\theta \right| = &\left| \int_{-\infty} ^{x} c(u) \partial_x u(t,y) dy \right| \notag \\
\leq  & \int_{\R} |c(u) \partial_x u(t,y)|dy \notag \\
\leq & -\int_{\R} \partial_t u(t,y)dy. \label{p2}
\end{align}
While, by the equation in (\ref{eq1}), we have
\begin{eqnarray} 
\dfrac{d}{dt}  \int_{\R}  \partial_t u(t,y)dy=0. \label{p3}
\end{eqnarray}
By (\ref{inicon3}), (\ref{p2}) and (\ref{p3}) we have
\begin{eqnarray} \label{p4}
\left|\int_0 ^{u(t,x)} c(\theta) d\theta \right|\leq -\int_{\R}  u_1 (x) dx <\int_{\theta_0} ^0 c(\theta) d\theta.
\end{eqnarray}
From  (\ref{p4}), (\ref{con3}) and  (\ref{con4}), we have (\ref{lemmazzes3}).
\end{proof}

\subsection*{Proof of Theorem \ref{mt1}}

From Lemma $\ref{zz3}$, $(\ref{b2})$ does not occur. 

The estimates (\ref{inftyes}) and  (\ref{lemmazzes3}) yield the uniform boundedness of  $\| \partial_x u\|_{L^{\infty}}$ and 
$\| \partial_t u\|_{L^{\infty}}$  with $t \in [0, T^{*})$. So $(\ref{b1})$ does not occur. 

Therefore, we complete the proof of Theorem $\ref{mt1}$.
  
 \hfill $\Box$ 
 \subsection*{Proof of Remark \ref{rem3}}
 
 Suppose $T^{*} < \infty$.
 
By a simple computation, we have
 $$
 \|u(t)\|_{L^{\infty}}\leq  \|u_0 \|_{L^{\infty}}+T^{*} \sup_{[0,T^{*})}\{ \|\partial_t u(t)\|_{L^{\infty}} \}.
 $$
 By Lemma \ref{zz3}, we obtain the boundedness of  $\|u(t)\|_{L^{\infty}}$,  $\|\partial_t u(t)\|_{L^{\infty}}$  and $\|\partial_x u(t)\|_{L^{\infty}}$ for $t, \in [0,T^{*}) $, which implies that  the blow up $(\ref{b1})$ and $(\ref{b2})$  does not occur, which is contradiction to $T^{*}< \infty$.

  \hfill $\Box$ 
 \section{Proof of Theorem \ref{mt2}}

First, we proof  $T^{*} <\infty$.
For this purpose, we use the following lemma.
\begin{Lemma}\label{sd} \it 
Suppose that $c(\theta ) \in C^{\infty}((\theta_0, \infty))$  and initial data $(u_0, u_1) \in H^{s+1}(\R) \times H^{s}(\R)$ with  $\displaystyle s>\frac{1}{2}$
satisfy   $(\ref{inicon1})$ and $(\ref{inicon4})$. Then the solution $u \in \bigcap_{j=0,1,2} C^{j}([0,T^{*});H^{s-j+1} (\R)) $
satisfies that
$$
\mbox{\rm supp\it}\, u(t,x) \subset [-c(0)t-K, c(0)t+K],
$$
where $K>0$ is a constant in  $(\ref{inicon4})$.
\rm
\end{Lemma} 
 Lemma \ref{sd}  is proved in many text book (e.g. p. 16 in Sogge's book \cite{sg}). Sogge Prove the same assertion as in Lemma \ref{sd}
 for the $C^2$ solution $u$. By the standard approximation argument,  Lemma \ref{sd} can be proved in the same way as in the proof in \cite{sg}.

Set $F(t)=-\int_{\R} u(t,x)dx$ for $0 \leq t < T^{*}$.

By the equation in  $(\ref{eq1})$,\ we have
\begin{eqnarray*}
\dfrac{d^2F}{dt^2} (t) = 0,
\end{eqnarray*}
which implies that
\begin{eqnarray} \label{F1}
F(t)=F(0)+tF'(0).
\end{eqnarray}
By Lemma \ref{zz1} and the fact that $u(t, \cdot ) >\theta_0$ for $t \in [0,T^{*})$, we have
\begin{align} 
F(t)=&-\int_{-c(0)t-K} ^{c(0)t+K} u(t,x) dx \notag \\
\leq & -\int_{-c(0)t-K} ^{c(0)t+K} \theta_0 dx \notag \\
=& -2\theta_0 (c(0)t+K). \label{F2}
\end{align}
From (\ref{F1}) and (\ref{F2}), we obtain that
\begin{eqnarray*} 
\frac{F(0) +2\theta_0 c(0)}{ \int_{\R} u_1 (x) dx -2\theta_0 c(0) } \geq t.
\end{eqnarray*}
We note that the left hand side of the above inequality is finite by (\ref{inicon5}).

Since $t$ can be chosen for all $[0,T^{*})$, we have $\displaystyle T^{*}\leq  \frac{F(0) +2\theta_0 c(0)}{ \int_{\R} u_1 (x) dx -2\theta_0 c(0) }<\infty$.

Next, we show that 
\begin{eqnarray} \label{b22}
\lim_{t\nearrow T^{*}} \inf_{(s,y)\in [0,t)\times \R} u(s,y)=\theta_0.
\end{eqnarray}
Suppose that (\ref{b22}) does not occur. 
So there exists a constant $\delta >0$ such that
$$
 c(u(t,x))\geq \delta,
$$
for all  all $(t,x) \in [0,T^{*})\times \R$.

 By Lemma \ref{zz3}, we have the boundedness of  $\| \partial_t u (t) \|_{L^{\infty}}$  and $\| \partial_x u (t) \|_{L^{\infty}}$ on  $[0, T^{*}]$, which is contradiction 
to the fact that $T^{*} < \infty$. Hence we have (\ref{b22}).

Finally, we show that
\begin{eqnarray} \label{con}
\lim_{t\nearrow T^{*} } u(t,x_0)=\theta_0  \ \ \mbox{for some} \ x_0 \in \R.
\end{eqnarray}
Since  $u(t,x)$ is a monotone decreasing function of $t$ for fixed $x$, we have
\begin{align} \label{x1}
\lim_{t\nearrow T^{*} } \inf_{(s,y)\in [0,t)\times \R} u(s,y)=&\lim_{t\nearrow T^{*} } \inf_{x \in \R} u(t,x) \notag \\
=& \inf_{x \in \R} \lim_{t\nearrow T^{*} }u(t,x).
\end{align}
The right hand side of (\ref{x1}) is equivalent to (\ref{con}) since   $\lim_{t\nearrow T^{*} }u(t,x)$ is compactly supported.

\hfill $\Box$

\begin{Remark}
The same theorem as Theorem $\ref{mt2}$  holds for the equation (\ref{raeq}) for $0\leq \lambda \leq 2$.
\end{Remark}
 
\section{Proof of Theorem $\ref{mt3}$}

We define  functions $R_1,\ R_2$ and characteristic lines $x_{\pm}$ as (\ref{ri}) and (\ref{chl}) respectively.

By $u_1 (x) \not \equiv 0$, we have  $R_1 (0,\cdot )\not\equiv 0$ or $R_2 (0,\cdot )\not\equiv 0$. We assume that $R_1 (0, x_0 )\not =0$.

Suppose that $T^{*} =\infty$.

From 
\begin{eqnarray} \label{y0}
\dfrac{d}{dt} u(t,x_{-}(t))=R_2 (t, x_{-} (t)),
\end{eqnarray}
and the assumption $R_2 (0,x)\geq 0$ , Lemma \ref{zz1} yields that
$u (t, x_{-} (t))$ is a monotone increasing function with $t$. By (\ref{con2}), there exists a $\delta >0$ such that
\begin{eqnarray} \label{y1} 
c(u(t,x_{-} (t)))\geq \delta.
\end{eqnarray}
In the same way as in the proof of Lemma \ref{zz1}, we obtain
$$
R_2 (t, x_{+} (t))=0 \ \ \mbox{for} \ t\geq 0,
$$
with $x_{+}(0) \not\in \mbox{supp}\, R_2 (0,\cdot )$.

Since $R_2 (0,\cdot )$ is compactly supported, there exists $T_0 >0$ such that
\begin{eqnarray} \label{y2}
R_2 (t, x_{-} (t))=0 \ \ \mbox{for} \ t\geq T_0 .
\end{eqnarray}
By (\ref{y0}) and (\ref{y2}), we have
\begin{eqnarray} \label{y3}
u (0,x_{-}(0))\leq u(t, x_{-}(t))\leq C,
\end{eqnarray}
for some constant $C>0$.

By (\ref{inicon6}), (\ref{y1}) and (\ref{y3}), we obtain
\begin{eqnarray} \label{y4}
\delta \leq c(u(t,x_{-} (t)))\leq C_1\ \ \mbox{and} \ \ 
C_2 \leq c'(u(t,x_{-} (t)))\leq C_3
\end{eqnarray}
for some constant $C_j >0$ for $j=1,\ 2$ and $3$.

We chose $x_{-}(0)$ such that $R_1 (0,x_{-}(0))>0$.
 
Noting that $R_1 (t, x_{-} (t))>0$ for $t\geq 0$, by (\ref{y2}) and (\ref{y3}), $R_1 (t, x_{-}(t))$ satisfies that
\begin{eqnarray}
\dfrac{d}{dt} R_1(t,x_{-} (t))\geq  CR_1 (t,x_{-}(t))^2,\ \ \mbox{for} \ t\geq T_0.
\end{eqnarray}
From $R(T_0, x_{-} (T_0))>0$, $R(t,x_{-}(t))$ is going to infinity in finite time, which is contradiction to $T^{*} = \infty$.

Since the first estimate in (\ref{y4}) holds on $[0,T^{*})$,  we have
\begin{eqnarray*}  
\varlimsup_{t\nearrow T^{*}} \| \partial_t u(t)\|_{L^{\infty}}+\|\partial_x u(t) \|_{L^{\infty}}=\infty .
\end{eqnarray*}

 \hfill $\Box$ 
 
 \section*{Acknowledgments}
I would like to express my gratitude to Professor Keiichi Kato for many valuable comments and warm encouragement.

\end{document}